\documentclass[11pt, a4paper]{amsart}

\usepackage{amsfonts}
\usepackage{amsthm}
\usepackage{amstext}
\usepackage{amsmath}
\usepackage{amscd}
\usepackage{amssymb}
\usepackage[mathscr]{eucal}
\usepackage{url}
\usepackage{enumitem}

\newtheorem{theorem}{Theorem}
\newtheorem{corollary}[theorem]{Corollary}
\newtheorem{proposition}[theorem]{Proposition}
\newtheorem{lemma}[theorem]{Lemma}

\theoremstyle{definition}
\newtheorem{definition}[theorem]{Definition}
\newtheorem{remark}[theorem]{Remark}

\numberwithin{equation}{section}

\newcommand\mL{L\kern-0.08cm\char39}

\newcommand{\Int}{{\rm Int}}

\newcommand{\CC}{\mathcal{C}}

\newcommand\Orb{{\rm Orb}}

\begin{document}

%%% Article information
\title[One scrambled pair implies Li-Yorke chaos]{For graph maps, one scrambled pair implies Li-Yorke chaos}

\date{May 16, 2012}

%%% Author information
\author[S. Ruette]{Sylvie Ruette}
\address{Laboratoire de Math\'ematiques, B\^atiment 425, CNRS UMR 8628, Universit\'e Paris-Sud 11, 91405 Orsay cedex, France}
\email{Sylvie.Ruette@math.u-psud.fr}

\author[\mL . Snoha]{L$\!$'ubom\'\i r Snoha}
\address{Department of Mathematics, Faculty of Natural Sciences,
            Matej Bel University, Tajovsk\'eho 40, 974 01 Bansk\'a Bystrica,
            Slovakia}
\email{Lubomir.Snoha@umb.sk}

\thanks{Most of this work was done while the second author was visiting Universit\'e Paris-Sud 11 at Orsay. The invitation and the kind hospitality of this institution are gratefully acknowledged. The second author was also partially supported by VEGA grant 1/0978/11 and by the Slovak Research and Development Agency under the contract No.~APVV-0134-10.}

%    General info


\begin{abstract}
For a dynamical system $(X,f)$, $X$ being a compact metric space with metric $d$
and $f$ being a continuous map $X\to X$, a set $S\subseteq X$
is scrambled if every pair $(x,y)$ of distinct points in $S$ is scrambled,
i.e., $\liminf_{n\to+\infty}d(f^n(x),f^n(y))=0$ and
$\limsup_{n\to+\infty}d(f^n(x),f^n(y))>0$. The system $(X,f)$ is
Li-Yorke chaotic if it has an uncountable scrambled set.  It is
known that, for interval and circle maps, the existence of a scrambled pair
implies Li-Yorke chaos, in fact the existence of a Cantor scrambled set.
We prove that the same result holds for graph maps.
We further show that on compact countable metric spaces
one scrambled pair implies the existence of an infinite scrambled set.
\end{abstract}

\maketitle

%%%%%%%%%%%%%%%%%%%%%%%%%%%%%%%%%%%%%%%%%%%%%%%%%%%%%%%%%%%%%%%%%%%

\section{Introduction and main results}\label{S:intro}

A \emph{(topological) dynamical system} is a pair $(X,f)$ or, less formally, a map $f: X\to X$,
where $X$ is a compact metric space and $f\colon X\to X$ is continuous.
The distance in any metric space will be denoted by $d$.

Let $f\colon X\to X$ be a dynamical system. The \emph{orbit} (under $f$) of a set $A\subseteq X$ is
$\Orb_f(A):=\bigcup_{n\ge 0}f^n(A)$, and the orbit $\Orb_f(x)$ of a point $x\in X$ is simply equal
to $\Orb_f(\{x\})$. The sequence $(f^n(x))_{n=0}^\infty$ is the \emph{trajectory} of $x$.
The \emph{$\omega$-limit set} (under $f$) of a point $x$, denoted by $\omega_f(x)$, is the set of all limit points
of the trajectory of $x$. It is a closed set and $f(\omega_f(x))=\omega_f(x)$. An $\omega$-limit set $\omega_f(x)$ is called \emph{orbit-enclosing} if there exists $y\in \omega_f(x)$ (hence $\Orb_f(y) \subseteq \omega_f(x)$) with $\omega_f(y)= \omega_f(x)$.

\begin{definition}
Let $f\colon X\to X$ be a dynamical system.  If
$x,y\in X$ and $\delta>0$, $(x,y)$ is a \emph{$\delta$-scrambled pair}
if
$$
\liminf_{n\to+\infty}d(f^n(x),f^n(y))=0\quad\text{ and }\quad
\limsup_{n\to+\infty}d(f^n(x),f^n(y))\ge \delta,
$$
and $(x,y)$ is a scrambled pair if it is $\delta$-scrambled for some
$\delta>0$.  A set $S\subseteq X$ is \emph{$\delta$-scrambled}
(resp. \emph{scrambled}) if for all $x,y\in S, x\neq y$, $(x,y)$ is a
$\delta$-scrambled (resp. scrambled) pair.  The dynamical system
$(X,f)$ is \emph{Li-Yorke chaotic} if there exists an uncountable
scrambled set in $X$.
\end{definition}

There are compact metric spaces which do not admit continuous selfmaps with scrambled pairs, say finite spaces and rigid spaces (a space is \emph{rigid} if it does not admit any continuous selfmap except of the identity and the constant maps). A less trivial example of such a space is the subspace $\{0\}\cup \{1,1/2,1/3,\dots\}$ of the real line (an easy proof is left to the reader).

On the other hand, every metric space containing an \emph{arc}, i.e. a homeomorphic copy of the real compact interval $I=[0,1]$, admits a continuous selfmap having scrambled pairs. This follows from the following facts:
1) $I$ admits such a map, 2) $I$ is an absolute retract for the class of all metric spaces~\cite{Bor} (recall that a subspace $S$ of $X$ is called a \emph{retract} of $X$ if there exists a retraction of $X$ onto $S$, i.e. a continuous map $r: X\to S$ such that $r(s)=s$ for all $s\in S$, and a metric space $A$ is called an \emph{absolute retract} for the class of all metric spaces, if for every metric space  $X$, every subspace of $X$ homeomorphic with $A$ is a retract of $X$) and 3) If $S\subseteq X$ are compact metric spaces with $S$ being a retract of $X$ and $S$ admitting a continuous selfmap with scrambled pairs then also $X$ admits a continuous selfmap with scrambled pairs (say the composition of a retraction $X\to S$ with such a selfmap of $S$).

If a system has a scrambled pair, it may happen that it has no scrambled set with three points. An example of a triangular map in the square with this property is found in~\cite{FPS}. Using the fact that the square (hence also the disc) is an absolute retract for the class of all metric spaces, we easily get that every metric space containing a subset homeomorphic with $2$-dimensional disc, in particular every manifold with dimension $\geq 2$, admits a continuous selfmap having a scrambled pair but no scrambled set with more than two points. In~\cite{GL} it was shown that the Cantor set and the Warsaw circle also admit continuous self-maps with this property. Examples of symbolic systems with only boundedly finite, or countable, scrambled sets were given in~\cite{BDM} --- symbolic systems generated by primitive constant-length substitutions have at most finite scrambled sets (some have no scrambled sets at all and others have finite scrambled sets); systems with infinite but only countable scrambled sets are obtained as inverse limits of a sequence of constant-length substitution systems. In~\cite{BHuS} it is shown that for every nonempty subset of the set $\{2,3,\dots \} \cup \{\aleph _0\} \cup \{c\}$, where $\aleph_0$ is the cardinal number of the set of positive integers and $c$ is the cardinal number of the reals, there is a dynamical system $(X, f)$ such that the set of cardinalities of all maximal scrambled sets of the system coincides with this set. Moreover, given any $\delta>0$, all scrambled sets of $(X, f)$ may be assumed to be $\delta$-scrambled and $X$ can be chosen to be an arc-wise connected one-dimensional planar continuum.

\smallskip

We are interested in sufficient conditions for the existence of a ``large" scrambled set in the sense of cardinality. Our main motivation is the following result proved by Kuchta and Sm\'\i tal.

\begin{theorem}[Kuchta, Sm\'\i tal \cite{KS}]\label{theo:scrambledpair-interval}
Let $I$ be a compact interval and $f\colon I \to I$ a continuous map.
If $f$ has a scrambled pair, then $f$ has a $\delta$-scrambled Cantor set for
some $\delta>0$.
\end{theorem}

It is difficult to prove this theorem from scratch. However, when one uses other
deep results from interval dynamics, it becomes obvious. In fact, it is an immediate
consequence of the dichotomy which holds for interval maps --- by~\cite{Sm}, cf.~\cite{JaSm},
any continuous map $f\colon I \to I$ satisfies one of the following two mutually exclusive properties:
\begin{enumerate}
  \item [(i)] $f$ has a $\delta$-scrambled Cantor set for some $\delta>0$;
  \item [(ii)] all trajectories of $f$ are \emph{approximable by cycles},
that is, for any $x$ and any $\varepsilon>0$ there is a periodic
point $p$ such that $\limsup_{n\rightarrow +\infty}
|f^{n}(x)-f^{n}(p)|<\varepsilon$.
\end{enumerate}

\smallskip

Later Kuchta~\cite{K} showed that the result from Theorem~\ref{theo:scrambledpair-interval} is true also for continuous selfmaps of the circle. In~\cite[p.315]{BHuS} the authors ask a question whether this is still true for graph maps. The main result of the present paper is an affirmative answer to this question.

Throughout the paper, a \emph{(topological) graph} is a non-degenerate compact
connected metric space $G$ containing a finite subset $V$ such that each
connected component of $G\setminus V$ is homeomorphic to an
open interval. A \emph{branching point} is a point having no neighborhood
homeomorphic to an interval (of any kind). The set of branching points is finite (it
is included in $V$).
A \emph{graph map} is a dynamical system on a graph, that is,
a continuous map $f\colon G\to G$, where $G$ is a graph.

In the present paper we prove the following theorem.

\begin{theorem}\label{theo:scrambedpairG}
Let $f\colon G\to G$ be a graph map. If $f$ has a scrambled pair, then
it has a Cantor $\delta$-scrambled set for some $\delta>0$.
\end{theorem}

By using easy arguments, one can extend this theorem to graphs which are not necessarily connected, i.e. to finite unions of connected graphs.

\smallskip

While the existence of a scrambled pair implies Li-Yorke chaos only in particular spaces, the following is a general result which holds on any compact metric space. By $h(f)$ we denote the topological entropy of $f$.

\begin{theorem}[Blanchard, Glasner, Kolyada, Maass~\cite{BGKM}]\label{theo:general-h>0}
Let $f\colon X\to X$ be a dynamical system. If $h(f)>0$, then $f$ has a $\delta$-scrambled Cantor set
for some $\delta>0$.
\end{theorem}

In~\cite{BGKM}, it is in fact proved only that positive topological entropy implies
the existence of a Cantor scrambled set. Although the result is not
stated in term of $\delta$-scrambled set, the proof of \cite[Theorem~2.3]{BGKM}
clearly implies Theorem~\ref{theo:general-h>0}.

Positive topological entropy is only a sufficient condition for the existence of a Cantor scrambled set, not a necessary one. Even on the interval there are zero entropy maps having Cantor $\delta$-scrambled sets~\cite{Sm}.

Theorem~\ref{theo:general-h>0} is a motivation for the question, what kind of a weaker assumption could be sufficient for the existence of an \emph{infinite} scrambled set.

Before giving an answer, we wish to bring the attention of the reader to the fact that
\begin{itemize}
\item for graph maps, positive topological entropy is equivalent with the existence of an infinite $\omega$-limit set containing a periodic point.
\end{itemize}
Indeed, if a graph map $f$ has an infinite $\omega$-limit set containing a periodic point, then it is a basic set in that classification of $\omega$-limit sets which we adopt in Section~\ref{S:prelim}, and so $f$ has positive entropy by Corollary~\ref{theo:basic-h>0}. Conversely, if a graph map $f: G\to G$ has positive entropy, then by~\cite[Theorem B]{LM} there are closed intervals $J, K \subseteq G$ with disjoint interiors, and $n\in \mathbb N$ such that $f^n(J)\cap f^n(K) \supseteq J\cup K$. Then, by~\cite[pp. 35--37]{BC}, $f$ has an infinite $\omega$-limit set containing a periodic point.

Contrary to the graph case, in general there is no relation between positive topological entropy and the existence of an infinite $\omega$-limit set containing a periodic point. On one hand, there exist positive entropy homeomorphisms which are minimal~\cite{R}, hence do not have any periodic points. On the other hand, the square and the topologist's sine curve admit zero entropy maps with an infinite $\omega$-limit set containing a fixed point~\cite{FPS},~\cite{GL}, though they have only scrambled pairs and no scrambled sets with three points.

In spite of the fact that in general there is no relation between positive topological entropy and the existence of an infinite $\omega$-limit set containing a periodic point, it turns out that the existence of some kinds of $\omega$-limit sets implies the existence of some kinds of scrambled sets. First observe the following trivial fact:
\begin{itemize}
\item if $f$ has an infinite $\omega$-limit set containing a periodic point, then there exist scrambled pairs in the system; however, it is possible that no scrambled set has cardinality larger than two.
\end{itemize}
By~\cite[Theorem 4.1]{HY2},
\begin{itemize}
\item if $X$ is an infinite compact metric space and $f\colon X\to X$  is transitive and has a periodic
point then there is an uncountable scrambled set for $f$.
\end{itemize}
Notice that an equivalent formulation of this result is:
\begin{itemize}
\item if $X$ is a compact metric space and a continuous map $f: X\to X$ has an infinite orbit-enclosing $\omega$-limit set containing a periodic point, then $f$ has an uncountable scrambled set.
\end{itemize}
This result and Theorem~\ref{theo:general-h>0} are independent. On one hand, positive entropy does not imply the existence of an infinite orbit-enclosing $\omega$-limit set containing a periodic point (again recall that there are positive entropy minimal homeomorphisms~\cite{R}) and, on the other hand, there exists a zero entropy transitive map with dense periodic points on the Cantor set~\cite{We}.

Our following result gives, again in terms of $\omega$-limit sets, a required sufficient condition for an infinite $\delta$-scrambled set.

\begin{proposition}\label{P:special omega}
Let $X$ be a compact metric space and $f:X\to X$ a continuous map. If $f$ has an infinite $\omega$-limit set containing a periodic point and containing also an isolated point (isolated in the relative topology of the $\omega$-limit set), then $f$ has an infinite $\delta$-scrambled set for some $\delta>0$.
\end{proposition}

In connection with this proposition and Theorem~\ref{theo:general-h>0} (cf. also the mentioned result from~\cite{HY2}), it would be interesting to find, if possible, a stronger assumption on $\omega$-limit sets than that in Proposition~\ref{P:special omega}, but \emph{weaker than positive topological entropy}, which would imply the existence of a Cantor scrambled set.

Using Proposition~\ref{P:special omega} one can prove the following corollary.

\begin{corollary}\label{C:cpct ctble}
Let $X$ be a compact \emph{countable} metric space and $f:X\to X$ a continuous map. If $f$ has a scrambled pair then
it has an infinite $\delta$-scrambled set for some $\delta>0$.
\end{corollary}

Systems on compact countable metric spaces are relatively simple (say, they always have zero entropy). However, recall that the first example of a completely scrambled compact system $(X,f)$, i.e. a system for which the whole space $X$ is a scrambled set, was found on a countable space~\cite{HY}.

\medskip

The paper is organized as follows. In Section~\ref{S:prelim} we study $\omega$-limit sets on graphs, which are the main tool in our proof of Theorem~\ref{theo:scrambedpairG} presented in Section~\ref{S:graphs}. In Section~\ref{S:infinite SS} we then prove Proposition~\ref{P:special omega} and Corollary~\ref{C:cpct ctble}.

\section{Preliminaries on $\omega$-limit sets on graphs}\label{S:prelim}

\subsection{Cycles of graphs and topological characterization of $\omega$-limit sets on graphs}\label{SS:cycles}
It is well known that a finite $\omega$-limit set is a periodic orbit.
In order to classify infinite $\omega$-limit sets, we need the notion of a cycle of graphs.

\begin{definition}\label{D:cycle}
Let $f\colon G\to G$ be a graph map.
A subgraph $K$ of $G$ is called \emph{periodic of period~$k$} or \emph{$k$-periodic} if $K, f(K), \dots, f^{k-1}(K)$ are pairwise disjoint and $f^k(K) = K$. If, instead of $f^k(K) = K$, it is known only that $f^k(K) \subseteq K$, the subgraph $K$ is called \emph{weakly periodic}. Then the set $\Orb (K)= \bigcup_{i=0}^{k-1} f^i(K)$ is called a \emph{cycle of graphs} if $K$ is periodic and a \emph{weak cycle of graphs} if $K$ is weakly periodic.
\end{definition}

\begin{remark}\label{rem:topo-omega}
To understand better what follows, it may be useful to recall the following fact (see~\cite{HM}). If $\omega$ is an $\omega$-limit set of a graph map, then $\omega$ is either a nonempty finite set (in fact, a periodic orbit), or an infinite closed nowhere dense set, or a finite union of non-degenerate subgraphs (which form a cycle of graphs). Conversely, whenever $\omega$ is a subset of a graph and is of one of the above three forms, then there exists a continuous selfmap of that graph such that $\omega$ is an $\omega$-limit set of it.
\end{remark}

\subsection{Cycles of graphs containing a given infinite $\omega$-limit set}\label{SS:infinite omega}

When classifying $\omega$-limit sets of one-dimensional systems,
the terminology and the definitions of the different classes vary
in the literature. Sharkovskii, when studying
interval maps, says that an $\omega$-limit set is
of genus $0$, or genus 1/first kind, or genus 2/second kind, if it
is finite, or infinite and containing no periodic point, or infinite and
containing a periodic point respectively. However, we will rather follow the terminology of \cite{B3},
see also \cite{B-book}. (Let us remark that the classification in \cite{HM}
is equivalent, although the equivalence is not straightforward and does not
seem to be explicitly proved in the literature. Let us also remark that
the classification in \cite{HM} concerns only maximal $\omega$-limit sets (with respect to the inclusion).
This makes sense because, for a graph map, any $\omega$-limit set is included in a maximal
$\omega$-limit set. This follows from the fact that the family of all $\omega$-limit sets of a graph map is closed with
respect to the Hausdorff metric in the hyperspace of all closed subsets of the graph, see~\cite{MS2}.)

We start with a lemma about cycles containing an $\omega$-limit set.

\begin{lemma}\label{lem:cycle}
Let $f: G \to G$ be a graph map and $\omega_f(x)$ be an infinite $\omega$-limit set.
\begin{enumerate}
\item
If $X$ is a weak $k$-cycle of graphs containing $\omega_f(x)$, there exists
a $k$-cycle of graphs $X'$ such that $\omega_f(x)\subseteq X'\subseteq X$.
\item
If $X,Y$ are two cycles of graphs containing  $\omega_f(x)$, there exists a cycle of graph
$Z$ such that $\omega_f(x)\subseteq Z\subseteq X\cap Y$. Moreover, the period
of $Z$ is greater than or equal to the maximum of the periods of $X$ and $Y$.
\item If $X$ is a $k$-cycle of graphs containing $\omega_f(x)$, then one
can write $X=\bigcup_{i=0}^{k-1}f^i(K)$, where $K$ is a
$k$-periodic graph, such
that, $\forall 0\le i<k$, $\omega_f(x)\cap f^i(K)=
\omega_{f^{k}}(f^i(x))$. Moreover, $\forall 0\le i<k$, the set
$\omega_f(x)\cap f^i(K)$ is infinite.
\end{enumerate}
\end{lemma}

\begin{proof}
(i) Suppose that $\omega_f(x)$ is a subset of $X=\Orb_f(K)$, where $K$ is
a weak $k$-periodic graph.
The set $K':=\bigcap_{n\ge 0}f^{nk}(K)$ is non-empty, compact and connected because
it is a decreasing intersection of non-empty compact connected sets, and
$f^k(K')=K'$. Moreover, $\omega_f(x)\subseteq \Orb_f(K')$ because
$f(\omega_f(x))=\omega_f(x)$. Since $\omega_f(x)$ is infinite, $K'$ is
non-degenerate, and thus $X':=\Orb_f(K')$ is a $k$-cycle of graphs.

(ii) Suppose that $\omega_f(x)\subseteq X\cap Y$, where $X,Y$ are cycles of
graphs. Let $Z_1,\ldots, Z_n$ denote the connected components of
$X\cap Y$ meeting $\omega_f(x)$ (they are finitely many because
$X\cap Y$ has finitely many connected components). For every $1\le i\le n$,
$f(Z_i)$ is included in a connected component of $X\cap Y$ and
meets $f(\omega_f(x))=\omega_f(x)$, and thus there exists
$\tau(i)\in \{1,\ldots, n\}$ such that $f(Z_i)\subseteq Z_{\tau(i)}$.
This implies that $Z_i$ is eventually
weakly periodic under $f$. Then it follows easily from the
properties of $\omega$-limit sets that $Z_1,\ldots, Z_n$ is actually
a weak cycle of graphs. By (i), there exists a cycle of graphs $Z$ such that
$\omega_f(x)\subseteq Z\subseteq Z_1\cup\cdots \cup Z_n\subseteq X\cap Y$.
The period of $Z$ is trivially greater than or equal to
the maximum of the periods of $X$ and $Y$.

(iii) Suppose that $J$ is a $k$-periodic graph such that $\omega_f(x)
\subseteq \Orb_f(J)=X$, and write $J_i=f^i(J)$ for all $0\le i<k$. At least one
of the sets $(\omega_f(x)\cap J_i)_{0\le i<k}$ is infinite, and so
$\omega_f(x)\cap J_i$ is infinite for every $0\le i<k$ because
$f(\omega_f(x)\cap J_i)=\omega_f(x)\cap J_{i+1\bmod k}$.
For all $0\le i<k$, choose pairwise disjoint neighborhoods $U_i$ of $J_i$
and let $V_i \subseteq U_i$ be a neighborhood of $J_i$ such that $\overline{V_i} \subseteq U_i$ and
$f(V_i)\subseteq U_{i+1\bmod k}$. Since $V:= \bigcup_{i=0}^{k-1} V_i$ is a
neighborhood of $\omega_f(x)$, there is $N$ such that $f^n(x)\in V$ for all $n\geq N$.
Choose $n_0\geq N$, a multiple of $k$, say $n_0=Nk$. Consider $j$ with $y:=f^{n_0}(x)\in V_j$.
Then $f^r(y)=f^{n_0+r}(x) \in V_{j+r\bmod k}$ for every nonnegative $r$. Hence
$\omega_{f^k}(f^i(y))\subseteq \overline{V_{j+i\bmod k}}\subseteq U_{j+i\bmod k}$ for all $0\le i<k$. This together with
$$
\Orb_f(J) \supseteq \omega_f(x) = \omega_f(y) = \bigcup_{i=0}^{k-1}\omega_{f^k}(f^i(y))
$$
gives that $\omega_{f^k}(f^i(y))\subseteq J_{j+i\bmod k}$ for all $0\le i<k$. Thus we have
$$
\omega_{f^k}(f^i(x)) = \omega_{f^k}(f^{n_0+i}(x)) = \omega_{f^k}(f^i(y)) \subseteq J_{j+i\bmod k}
$$
for all $0\le i<k$. Then for $K:=J_j$ we get $\omega_{f^k}(f^i(x)) \subseteq f^i(K)$ and so
$\omega_f(x)\cap f^i(K)= \omega_{f^{k}}(f^i(x))$ for all $0\le i<k$. This proves (iii).
\end{proof}

For an infinite $\omega$-limit set $\omega_f(x)$, let
\begin{equation}\label{Eq:cycles}
\CC(x):= \{X\mid X\subseteq G \text{ is a cycle of graphs and } \omega_f(x) \subseteq X \}.
\end{equation}
Since the whole graph $G$ is weakly $1$-periodic, Lemma~\ref{lem:cycle}(i)
implies that $\CC(x)$ is never empty.

The periods of the cycles in $\CC(x)$ are either unbounded or
bounded. We study these two possibilities in the next subsections.

\subsection{Solenoid $\omega$-limit sets}\label{SS:solenoid}

Next lemma describes the situation when the periods of the cycles in
$\CC(x)$ are unbounded.

\begin{lemma}\label{lem:solenoid}
Let $f$ be a graph map and let $\omega_f(x)$ be an infinite $\omega$-limit
set such that the periods of the cycles in $\CC(x)$ are not bounded. Then
there exists a sequence of cycles of graphs $(X_n)_{n\ge 1}$ with
strictly increasing periods $(k_n)_{n\ge 1}$ such that, for all $n\ge 1$,
$X_{n+1}\subseteq X_n$ and $\omega_f(x)\subseteq \bigcap_{n\ge 1} X_n$.
Moreover, for all $n\ge 1$,
$k_{n+1}$ is a multiple of $k_n$ and
every connected component of $X_n$ contains the same number (equal
to $k_{n+1}/k_n\ge 2$) of components of $X_{n+1}$.
Furthermore, $\omega_f(x)$ contains no periodic point.
\end{lemma}

\begin{proof}
By assumption, there exists a sequence $(Y_n)_{n\ge 0}$ of cycles of graphs
in $\CC(x)$, with strictly increasing periods $(l_n)_{n\ge 1}$.
We define inductively a sequence $(Y'_n)_{n\ge 1}$ as follows. Let
$Y'_1=Y_1$. If $Y'_n$ is already defined then, according to
Lemma~\ref{lem:cycle}(ii),
there exists a $l'_{n+1}$-cycle of graphs $Y'_{n+1}$ such that
$\omega_f(x)\subseteq Y'_{n+1}\subseteq Y'_n\cap Y_{n+1}$ and $l'_{n+1}\ge
l_{n+1}$. Choose a subsequence $(n_i)_{i\ge 1}$ such that
$\forall i\ge 1, l'_{n_{i+1}}> l'_{n_i}$, and set $X_i:=Y'_{n_i}$. Then
$(X_i)_{i\ge 1}$ is a sequence of cycles of graphs containing
$\omega_f(x)$, with strictly increasing periods and such that $\forall i\ge 1$,
$X_{i+1}\subseteq X_i$. The fact that $k_{i+1}$ is a multiple of $k_i$
follows trivially from the inclusion of the cycles. Fix $i$.
We can write $X_i=\Orb_f(K_i)$ and $X_{i+1}=\Orb_f(K_{i+1})$ with
$K_i, K_{i+1}$ periodic intervals such that $K_{i+1}\subseteq K_i$. Let
$p:=k_{i+1}/k_i$. Then, for a given $0\le j<k_i$, the sets
$(f^{nk_i+j}(K_{i+1}))_{0\le n<p}$ are pairwise disjoint and included in
$f^j(K_i)$. This shows that every $f^j(K_i)$ (which is a connected component
of $X_i$) contains $p$ components of $X_{i+1}$.
Finally, if $z$ is a $k$-periodic point in $X_n$, then $k\ge k_n$, and so
$\bigcap_{n\ge 1}X_n$ contains no periodic point.
\end{proof}

\begin{definition}
An infinite $\omega$-limit set $\omega_f(x)$ of a graph map is called a
\emph{solenoid} if the periods of cycles in $\CC(x)$ are not bounded.
\end{definition}

Notice that, for instance according to Remark~\ref{rem:topo-omega},
the solenoid $\omega_f(x)$ is necessarily nowhere dense.

\subsection{$\omega$-limit sets in a minimal cycle of graphs}

Next lemma states that there exists a minimal cycle containing
$\omega_f(x)$ when the periods of the cycles in $\CC(x)$ are bounded.

\begin{lemma}\label{lem:mincycle}
Let $f$ be a graph map and let $\omega_f(x)$ be an infinite
$\omega$-limit set such that the periods of the cycles in $\CC(x)$ are bounded.
There exists a cycle of graphs $X\in \CC(x)$ such that,
$\forall\, Y\in \CC(x)$, $X\subseteq Y$. The period of $X$ is maximal
among the periods of all cycles in  $\CC(x)$.
\end{lemma}

\begin{proof}
Let $k$ denote the maximal period of the cycles in $\CC(x)$, and define
$$
\CC_k:= \{X\in \CC(x) \mid X \text{ is of period }k\}.
$$
Let $(Y_{\lambda})_{\lambda\in \Lambda}$ be a totally ordered
family in $\CC_k$ (that is, all elements in $\Lambda$ are comparable
and, if $\lambda\le \lambda'$, then $Y_{\lambda}\subseteq Y_{\lambda'}$).
Then $Y=\bigcap_{\lambda\in\Lambda}Y_{\lambda}$ is compact and has
$k$ connected components because this is a decreasing intersection of
$k$-cycles, and
$f(Y)=Y$. Moreover, $\omega_f(x)\subseteq Y$. Hence $Y\in\CC_k$.
Thus Zorn's Lemma applies, and there exists an element $X\in \CC_k$
such that,
\begin{equation}\label{eq:mincycle}
\forall X'\in\CC_k,\ X'\subseteq X\Rightarrow X'=X.
\end{equation}
Let $Y\in\CC(x)$. By Lemma~\ref{lem:cycle}(ii), there exists
$Z\in\CC(x)$ such that $Z\subseteq X\cap Y$, and the period
of $Z$ is greater than or equal to the period of $X$. On the other hand,
the period of $Z$ is at most $k$ by definition. Hence $Z\in\CC_k$.
Then $Z=X$ by \eqref{eq:mincycle}, i.e., $X\subseteq Y$.
\end{proof}

\smallskip

If $X$ is any finite union of subgraphs of $G$ such that $f(X)\subseteq X$, we define
\begin{equation}\label{eq:E(X,f)-1}
E(X,f)=\{y\in X\mid \forall\, U\text{ neighbourhood of $y$ in }X,\,
\overline{\Orb_f(U)}=X\}.
\end{equation}
This set is obviously closed and it is easily seen to be $f$-invariant, i.e., $f(E(X,f)) \subseteq E(X,f)$.
In general $f(E(X,f)) \neq E(X,f)$. For instance, let $G$ be a circle and $X=G$. The circle $G$ is the union
of ``western half-circle" and ``eastern half-circle". Let $f$ restricted to any of these half-circles be
topologically conjugate to the tent map, the ``south pole" of $G$ being a fixed point of $f$ and the
``north pole" being mapped to the ``south pole". Then $E(X,f)$ consists of the two ``poles" but $f(E(X,f))$
is a singleton containing just the ``south pole".

Notice also that $E(X,f)$ was defined without referring to any $\omega$-limit set
and that
\begin{equation}\label{eq:E(X,f)-2}
\text{if \, $E(X,f)\neq\emptyset$, \, then \, $f(X)=X$}.
\end{equation}

\begin{lemma}\label{lem:prol}
Let $f\colon G\to G$ be a graph map and
let $\omega_f(x)$ be an infinite $\omega$-limit set
such that the periods of cycles of graphs
containing $\omega_f(x)$ are bounded, and let $K$ be the minimal cycle of
graphs containing $\omega_f(x)$.
\begin{enumerate}
\item
For every $y\in \omega_f(x)$ and for every
relative neighborhood $U$ of $y$ in $K$, $\overline{\Orb_f(U)} = K$.
\item
$\omega_f(x)\subseteq E(K,f)$. In particular, $E(K,f)$ is infinite.
\end{enumerate}
\end{lemma}

\begin{proof}
Since $K$ is the union of finitely many subgraphs, the set $\partial K$ is
finite. Hence $\omega_f(x)\cap \Int(K)\neq\emptyset$. This implies
that there exists $N$ such that $f^N(x)\in\Int(K)$, and thus
$\Orb_f(f^N(x))\subseteq K$. Then $\omega_f(x)=\omega_f(f^N(x))$ is an
$\omega$-limit set for the restricted map $f|_K$.
Let $y\in \omega_f(x)$ and $U$ be a relative neighborhood of $y$ in $K$.
Then, by considering the map $f|_K$, we get that there exist
$i\ge N$ and $n\ge 1$ such that $f^i(x), f^{i+n}(x)\in U$, which implies
that $f^n(U)\cap U\neq\emptyset$. It follows that the set $X:=\overline{\Orb_f(U)}$
has at most $n$ connected components, and is a weak cycle
of graphs. Moreover, $X\supseteq \omega_f(x)$. Then one can find a cycle of
graphs which contains $\omega_f(x)$ and is included in $X$ by
Lemma~\ref{lem:cycle}(i).
By minimality of $K$, this cycle is equal to $K$. Hence $X=K$, which is (i).
Then (ii) trivially follows from (i).
\end{proof}

Notice that the definition of $E(X,f)$ is unfortunately missing in \cite{B1}.
In Section~2 of \cite{B3}, the definition is given and the theorem is recalled
(the series of papers \cite{B1, B2, B3} forms a whole).
Notice that in Blokh's papers, a ``graph'' (also called a one-dimensional
ramified manifold) is not assumed to be connected, and is actually a
finite union of graphs with the definition of a graph we use.

Theorem~\ref{theo:Blokh-conjugacyE} is \cite[Theorem 2]{B1} (stated in the
slightly restricted case of a graph). To state it, we need the notion of
almost conjugacy.

\begin{definition}
Let $f\colon X\to X$, $g\colon Y\to Y$ be two continuous maps. A continuous
map $\varphi\colon X\to Y$ is a \emph{semi-conjugacy} between $f$ and
$g$ if $\varphi$ is onto and $\varphi\circ f=g\circ \varphi$. If in
addition $\varphi$ is a homeomorphism, then it is a \emph{conjugacy}
between $f$ and $g$.

Assume further that $K\subseteq X$ is a closed set such that $f(K)\subseteq K$.
Then a semi-conjugacy $\varphi$ between $f$ and $g$ is an \emph{almost conjugacy}
between $f|_K$ and $g$ if
\begin{enumerate}
\item [(i)] $\varphi(K)=Y$,
\item [(ii)] $\forall y\in Y$, $\varphi^{-1}(y)$ is connected,
\item [(iii)] $\forall y\in Y$, $\varphi^{-1}(y)\cap K=\partial \varphi^{-1}(y)$,
where $\partial$ denotes the boundary in $X$,
\item  [(iv)] $\exists N\ge 1$ such that, $\forall y\in Y$, $\varphi^{-1}(y)\cap K$
has at most $N$ elements (and, by (i), at least one element).
\end{enumerate}
\end{definition}

Notice that, when $f$ and $g$ are graph maps, the last condition of this
definition  is implied by the other ones. Indeed,
if the set $\varphi^{-1}(y)$ is connected,  then it is either a singleton
or a subgraph of the graph $X$,
and thus the cardinality of $\partial \varphi^{-1}(y)$ is finite and uniformly bounded.

Recall that if an $\omega$-limit set $\omega_f(x)$ contains a point $y$ (hence contains also its orbit) such that $\omega_f(y) = \omega_f(x)$, we say that the $\omega$-limit set $\omega_f(x)$ is orbit-enclosing. Recall also that a set is called \emph{perfect} if it is closed and dense in itself.

\begin{theorem}[Blokh~\cite{B1}]\label{theo:Blokh-conjugacyE}
Let $f\colon G\to G$ be a graph map and $X\subseteq G$ a finite union of
subgraphs such that $f(X)\subseteq X$.
Suppose that $E=E(X,f)$ is infinite. Then $E$ is a perfect
set, $f|_E$ is transitive (i.e., $E$ is an orbit enclosing $\omega$-limit set)
and, $\forall z\in G$, if $\omega_{f}(z)\supseteq E$ then $\omega_{f}(z)=E$
(hence, $E$ is a maximal $\omega$-limit set).
Moreover, there exists a transitive map $g\colon Y\to Y$, where $Y$ is
a finite union of graphs, and a semi-conjugacy $\varphi\colon X\to Y$
between $f|_X$ and $g$ which almost conjugates $f|_E$ and $g$.
\end{theorem}

\begin{remark}
In \cite[Theorem 2]{B1}, we may understand  that $E=E(X,f)$ is a maximal
$\omega$-limit set for the restricted map $f|_X$ (when considering $z$
such that $\omega_f(z)\supseteq E$, it is not stated whether $z$ belongs
to $G$ or $X$). But it is easy to show that if $E$ is a maximal
$\omega$-limit set for $f|_X$, then it is also a maximal
$\omega$-limit set for $f$. Indeed, suppose that $E$ is a maximal
$\omega$-limit set for $f|_X$ and let $z\in G$ be such that
$\omega_f(z)\supseteq E$. The set $E$ is infinite and $\partial X$ is
finite, and so $E\cap\Int(X)\neq\emptyset$, which implies that
$z':=f^n(z)\in X$ for some $n$. Then $\Orb_f(z')\subseteq X$ and, since
$\omega_f(z')=\omega_f(z)$, we get that $\omega_f(z)=E$.
\end{remark}

\begin{remark}\label{R:mincycle}
In Theorem~\ref{theo:Blokh-conjugacyE}, $X$ is in fact the minimal cycle of graphs containing the infinite $\omega$-limit set $E(X,f)$. To show this, first realize that, by~(\ref{eq:E(X,f)-1}) and~(\ref{eq:E(X,f)-2}), $X$ is indeed a cycle of graphs containing $E(X,f)$. Further, let $Y$ be a cycle of graphs containing $E(X,f)$. We claim that $X\subseteq Y$. Suppose this is not the case. Then, by Lemma~\ref{lem:cycle}(ii), there is a cycle of graphs $Z\subseteq X\cap Y \subsetneq X$ containing $E(X,f)$. Since $E(X,f) \subseteq Z \subsetneq X$ is infinite and $Z, X$ are cycles of graphs, there is a point $y\in E(X,f)$ such that some neighborhood $U$ of $y$ in $Z$ is also a neighborhood of $y$ in $X$. Then, by~(\ref{eq:E(X,f)-1}),  $\overline{\Orb_f(U)}=X$. However, $U\subseteq Z$ and so $\overline{\Orb_f(U)}\subseteq Z$. We get $X\subseteq Z$ which contradicts the fact that $Z\subsetneq X$.
\end{remark}

If $\omega_f(x)$ is infinite and included in a minimal cycle
of graphs $K$, then the set $E(K,f)$ contains $\omega_f(x)$ and so is infinite by Lemma~\ref{lem:prol}(ii).
Therefore, in such a case Theorem~\ref{theo:Blokh-conjugacyE} states that the
set $E(K,f)$ is a (maximal) $\omega$-limit set and also gives a partial
description of $f|_{E(K,f)}$.
This explains why the next definition is relevant in the classification
of $\omega$-limit sets.

\begin{definition}\label{D:basic}
Let $f\colon G\to G$ be a graph map and $X\subseteq G$ a finite union of
subgraphs such that $f(X)\subseteq X$. If $E(X,f)$ is infinite (hence, by
Theorem~\ref{theo:Blokh-conjugacyE}, it is an orbit enclosing $\omega$-limit
set of $f$ and it is a maximal $\omega$-limit set of $f$), it is called
a \emph{basic set} if $X$ contains a periodic point, and \emph{circumferential}
otherwise.
\end{definition}

Let us recall that in this definition, by Remark~\ref{R:mincycle}, $X$ is a minimal cycle of graphs containing $E(X,f)$.

Next result is due to Blokh \cite{B0} (see also \cite[Corollary 1]{B4}
for the statement without proof). A proof in English can be found in
\cite{ADR}.

\begin{theorem}[Blokh~\cite{B0}]\label{theo:transitive-h>0}
Let $Y$ be a finite union of graphs and
$g\colon Y\to Y$ be a transitive, continuous map.
If $g$ has periodic points, then it has positive topological entropy.
\end{theorem}

\begin{corollary}\label{theo:basic-h>0}
If a graph map $f$ admits a basic $\omega$-limit set,
then $h(f)>0$.
\end{corollary}

\begin{proof}
Let $E=E(X,f)$ be a basic set. According to Theorem~\ref{theo:Blokh-conjugacyE},
there exist a transitive map $g\colon Y\to Y$, where $Y$ is a finite union
of graphs, and a semi-conjugacy
$\varphi\colon X\to Y$ between $f|_X$ and $g$. The map $g$
has a periodic point because $X$ contains a  periodic point by definition. Then
$h(g)>0$ by Theorem~\ref{theo:transitive-h>0}, which implies that
$h(f)>0$.
\end{proof}

Now consider a circumferential set $E(X,f)$. Let $X_1,\ldots, X_k$ be the
connected components of $X$. Then $X$ is a cycle of graphs and, for every
$1\le i\le k$, $f^k|_{X_i}$ is ``almost'' an irrational rotation. More
precisely, either, for every $i$, $f^k|_{X_i}$ is conjugate to an irrational rotation
(and in this case $E(X_i,f^k)=X_i$), or, for every $i$,
there exists a semi-conjugacy $\varphi_i$ between $f^k|_{X_i}$ and an irrational
rotation which is an almost conjugacy on $f^k|_{E(X_i,f^k)}$, and
every connected component of $X_i\setminus E(X_i,f^k)$ is sent to one point
by $\varphi_i$. In the latter case the $\omega$-limit set $E(X,f)$ is called of \emph{Denjoy type}.
For precise statements and proofs, see \cite{B3} and \cite{MS}.
Notice that this situation cannot occur for interval or tree maps. We do  not discuss the circumferential
$\omega$-limit sets in more details, because we will only need the following result. It is due to Blokh~\cite{B0} (see \cite[Theorem S, p. 506]{B1} for a statement in English).

\begin{theorem}[Blokh~\cite{B0}]\label{theo:transitive-noperiodicpoint}
A transitive graph map with no periodic point is conjugate
to an irrational rotation of the circle.
\end{theorem}

%%%%%%%%%%
\section{On graphs, one scrambled pair implies Cantor scrambled set}\label{S:graphs}

This section is first of all devoted to the proof of Theorem~\ref{theo:scrambedpairG},
which is the main result of this paper. We recall this theorem.

\medskip\noindent\textbf{Theorem~\ref{theo:scrambedpairG}.}
Let $f\colon G\to G$ be a graph map. If $f$ has a
scrambled pair, then it has a Cantor $\delta$-scrambled set for some
$\delta>0$.

\begin{lemma}\label{lem:almostconj2}
Let $f\colon G\to G$, $g\colon
G'\to G'$ be two graph maps and  $E\subseteq G$ a closed set such that
$f(E)\subseteq E$. Suppose that $\varphi\colon G\to G'$ is a
semi-conjugacy between $f$ and $g$, which is an almost conjugacy
between $f|_E$ and $g$.  If $g$ is an irrational rotation of the circle, then $f$
has no scrambled pair.
\end{lemma}

\begin{proof}
Let $x,y$ be two points in $G$ such that
$$\liminf_{n\to+\infty}d(f^n(x),f^n(y))=0.$$
The semi-conjugacy implies that
$\liminf_{n\to+\infty}d(g^n(\varphi(x)),g^n(\varphi(y)))=0$.  This is
possible only if $\varphi(x)=\varphi(y)$ because $g$ is a rotation.
For every $n\ge 0$, we set $z_n:=\varphi(f^n(x))=\varphi(f^n(y))$ and
$G_n=\varphi^{-1}(z_n)$. Then $G_n$ is a closed connected set
containing both $f^n(x)$ and $f^n(y)$. If there exists $k$ such that
$G_k$ is reduced to a single point, then $f^k(x)=f^k(y)$ and the
trajectories of $x$ and $y$ eventually coincide. Otherwise, all the sets
$G_n$ are  subgraphs and so have non-empty interiors.  If there exist $m<n$
such that $G_m\cap G_n\neq\emptyset$, then
$z_m=z_n=g^{n-m}(z_m)$. This is impossible because $g$ has no periodic
point. Therefore, the subgraphs $(G_n)_{n\ge 0}$ are pairwise
disjoint. This implies that the diameter of $G_n$ tends to $0$ when
$n$ goes to $+\infty$. Hence
$$\lim_{n\to+\infty}d(f^n(x),f^n(y))=0,$$
and $f$ has no scrambled pair.
\end{proof}

\begin{proof}[Proof of Theorem~\ref{theo:scrambedpairG}]
Let $(x,y)$ be a scrambled pair. Then at least one of
the points $x,y$ has an infinite $\omega$-limit set (if both
$\omega$-limit sets were finite, they would be periodic orbits and  the pair
$(x,y)$ would not be scrambled).  Say, $\omega_f(x)$ is infinite.

Suppose first that $\omega_f(x)$ is a solenoid. By Lemma~\ref{lem:solenoid},
we have $\omega_f(x) \subseteq \bigcap_{n=1}^\infty X_n$ where $(X_n)_{n=1}^\infty$ is a
nested sequence of cycles of graphs whose periods $(k_n)_{n=1}^\infty$
form a strictly increasing sequence of positive integers.
One can choose $n$ such that $k_n\geq 2$ is larger than
the number of branching points of $G$ (which is finite), and thus some
connected component $I$ of $X_n$ is an arc which does not
contain any branching point of $G$. Then the arc $I$ contains at least
$4$ subarcs which are connected components of $X_{n+2}$. Hence there
is a component of $X_{n+2}$, call it $J$, which is a subset of the
interior of the arc $I$. Since $\omega_f(x)$ is infinite and
$f(\omega_f(x)) = \omega_f(x)$, it intersects each of the components
of $X_{n+2}$ in an infinite set by Lemma~\ref{lem:cycle}(iii).
Hence there is a point of
$\omega_f(x)$ which lies in the interior of $J$ and so, for some
$i_0$, $x':=f^{i_0}(x)$ also belongs to the interior of $J$.  Let
$g=f^{k_{n+2}}$. Then $g(J)=J$ and so, for all $j\ge 0$, $g^j(x')\in
J$. Let $y':=f^{i_0}(y)$. Since $(x,y)$ is a scrambled pair for $f$,
$(x',y')$ is a scrambled pair for $g$. In particular, there exists
$j_0$ such that $y'':=g^{j_0}(y')$ is so close to $x'':=
g^{j_0}(x')\in J$, that it belongs to $I$.  Then $g(I)=I$ and
$x'',y''\in I$ form a scrambled pair of the \emph{interval map}
$g|_I\colon I\to I$. By Theorem~\ref{theo:scrambledpair-interval},
$g|_I$ has a $\delta$-scrambled Cantor set for some $\delta>0$. This set
is $\delta$-scrambled also for $f$.

Suppose now that $\omega_f(x)$ is not a solenoid, and let $K$ be the
minimal cycle of graphs containing $\omega_f(x)$ given by Lemma~\ref{lem:mincycle}.
The set $E(K,f)$ is infinite by Lemma~\ref{lem:prol}.  If  $E(K,f)$ is a
basic set, then $h(f)>0$ by  Corollary~\ref{theo:basic-h>0}, and the
conclusion follows from  Theorem~\ref{theo:general-h>0}.
The proof of the theorem will be finished if we show that $E(K,f)$ is not circumferential.
Suppose on the contrary that  $K$ contains no periodic point. Let $K_1,\ldots, K_k$ be  the
connected components of $K$. It is clear that $E(K,f)=\bigcup_{i=1}^k
E(K_i,f^k)$. Since $\omega_f(x)$ is infinite, Lemmas~\ref{lem:cycle}(ii)
and \ref{lem:prol}(ii) imply  that each of the sets
$E(K_i,f^k)$ is infinite. Let $g:=f^k|_{K_1}$ and $E:=E(K_1,g)$.
According to Theorem~\ref{theo:Blokh-conjugacyE}, there exist a
transitive graph map $g'\colon G'\to G'$ and a semi-conjugacy
$\varphi$ between $g$ and $g'$ which is an almost conjugacy between
$g|_E$ and $g'$ (the set $G'$ is a graph because $G'=\varphi(K_1)$ and
$K_1$ is connected). Then $g'$ has no periodic point because $g$ has no
periodic point, and thus $g'$ is conjugate to an irrational
rotation by Theorem~\ref{theo:transitive-noperiodicpoint}.

Since $(x,y)$ is a scrambled pair, $\Omega:=\omega_f(x)\cap
\omega_f(y)\neq \emptyset$. Of course, $f(\Omega)\subseteq \Omega$ and
since $\omega_f(x) \subseteq K$, also $\Omega \subseteq K$.
Suppose that $\Omega \subseteq \partial K$. Then $\Omega$ is
finite and, being $f$-invariant, contains a periodic orbit of $f$,
which contradicts the assumption that $K$ contains no periodic point.
Therefore there exists $z\in \Omega \cap \Int(K)$. It follows that the
trajectories of $x$ and $y$ enter $K$. So, $f^n(x),f^n(y)\in K$ for
all sufficiently large $n$. Since $(x,y)$ is a scrambled pair and
the distances between the components of the cycle $K$ are positive,
there exists $N$ such that $x':=f^N(x)$ and $y':=f^N(y)$ belong
to $K_1$. But $(x',y')$ is then a scrambled pair for $g$, which
contradicts Lemma~\ref{lem:almostconj2}. This finishes the proof.
\end{proof}

\section{A special $\omega$-limit set implies an infinite scrambled set}\label{S:infinite SS}

In this section, we show Proposition~\ref{P:special omega} and
Corollary~\ref{C:cpct ctble}. For clarity, we restate these results before
their proof.

In a compact dynamical system $(X,f)$, an $\omega$-limit set $\omega_f(x)$ may contain a periodic point, hence a whole periodic orbit $P$. If $P$ is a proper subset of $\omega_f(x)$, then $\omega_f(x)$ is infinite (otherwise it would be a periodic orbit containing properly another periodic orbit, which is absurd) and, by~\cite{Shark}, no point of $P$ is isolated in $\omega_f(x)$.

By $B(x,\varepsilon)$ and $\overline{B}(x,\varepsilon)$ we denote the open and the closed, respectively, ball with center $x$ and radius $\varepsilon$.

\medskip\noindent\textbf{Proposition~\ref{P:special omega}.}
Let $X$ be a compact metric space and $f:X\to X$ a continuous map. If $f$ has an infinite $\omega$-limit set containing a periodic point and containing also an isolated point (isolated in the relative topology of the $\omega$-limit set), then $f$ has an infinite $\delta$-scrambled set for some $\delta>0$.

\begin{proof}
First we show that we may, without loss of generality, assume that the periodic point is a fixed point. To see this, suppose that $\omega_f(x)$ is infinite, $P\subseteq  \omega_f(x)$ is a periodic orbit of period $k>1$ and $z_0 \in \omega_f(x)$ is an isolated point of $\omega_f(x)$. Put $g:= f^k$. Then, as is well known (see e.g.~\cite[pp. 70-71]{BC}),
$\omega_{f}(x)=\bigcup_{i=0}^{k-1}\omega_{g}(f^{i}(x))$ and each of the sets in this union is mapped by $f$ onto the next one
$\bmod \,k$. It follows that each of the sets in the union is infinite and contains a point from $P$. One of them of course contains $z_0$ as an isolated point. Thus, the map $g$ has an infinite $\omega$-limit set containing a fixed point of $g$ and an isolated point. If we prove that $g$ has an infinite $\delta$-scrambled set, then it will be a $\delta$-scrambled set also for $f$.

So, assume that $\omega_f(x)$ is infinite, contains a fixed point $z$ of $f$ and an isolated point $z_0$. As we already know, $z_0 \neq z$. Choose $\delta>0$ such that $B(z_0, 4\delta) \cap \omega_f(x) = \{z_0\}$. In particular, $d(z,z_0)>4\delta$. Put $\overline{B}_3 := \overline{B}(z_0,3\delta)$ and $B_1 := B(z_0,\delta)$. Since $\omega_f(x)$ is infinite, any two points in the trajectory of $x$ under $f$ are distinct and $\Orb_f(x)$ is an infinite set. We are going to prove that it is a $\delta$-scrambled set for $f$ (in fact even $\{z\} \cup \Orb_f(x)$ is $\delta$-scrambled). To this end fix two points $x_2$ and $x_1 := f^m(x_2)$ in $\Orb_f(x)$, where $m$ is a positive integer.

Of course, $\omega_f(x_1) = \omega_f(x_2) = \omega_f(x) \ni z$. Since $z$ is a fixed point of $f$, for an arbitrarily small neighborhood of $z$ there exists $j$ such that both $f^j(x_2)$ and $f^{j+m}(x_2) = f^j(x_1)$ are in this neighborhood. Hence $x_1$ and $x_2$ are proximal (that is, $\liminf_{n\to+\infty}
d(f^n(x_1),f^n(x_2))=0$).

On the other hand, both $z_0$ and $z$ are in $\omega_f(x_1)$ and so the trajectory of $x_1$ visits $B_1$ infinitely many times and is outside $\overline{B}_3$ also infinitely many times. Taking further into account that $\omega_f(x_1) \cap (\overline{B}_3 \setminus B_1) =\emptyset$, we see that the trajectory of $x_1$ visits the compact set $\overline{B}_3 \setminus B_1$ only finitely many times and so it is eventually in $(X\setminus \overline{B}_3) \cup B_1$. Moreover, $x_1$ is proximal to the fixed point $z$ and so there are arbitrarily long intervals of consecutive times when the trajectory of $x_1$ is in  $X\setminus \overline{B}_3$. It follows that there are infinitely many times $j$ with $f^j(x_1) \in B_1$ and $f^{j-m}(x_1) = f^j(x_2) \in X\setminus \overline{B}_3$. For each such $j$ we have $d(f^j(x_1), f^j(x_2)) > 2\delta$. Hence $\limsup_{n\to\infty} d(f^n(x_1), f^n(x_2))\geq 2\delta >\delta$.
\end{proof}

The fact that every countable compact Hausdorff space has the periodic point property, is known~\cite{St}. Since there is a very short dynamical proof~\cite{Sar}, we repeat it.

\begin{lemma}\label{L:ppp}
Let $X$ be a countable compact Hausdorff space. Then every continuous map $f: X\to X$ has a periodic point.
\end{lemma}

\begin{proof}
Since $X$ is compact, there is a minimal set $M$ of the system $(X,f)$ (i.e. a minimal with respect to the inclusion, non-empty, closed, $f$-invariant subset of $X$).
Then $M$ is a compact Hausdorff, hence a Baire space. Since it is countable, applying Baire Category Theorem we get that it has an isolated point $z$. However, in the minimal system $(M,f|_M)$ every orbit is dense, therefore the isolated point $z$ is periodic (and $M$ is just the orbit of $z$).
\end{proof}

On a countable space, Li-Yorke chaos is impossible for cardinality reasons. However, the following holds.

\medskip\noindent\textbf{Corollary~\ref{C:cpct ctble}.}
Let $X$ be a compact \emph{countable} metric space and $f:X\to X$ a continuous map. If $f$ has a scrambled pair then
it has an infinite $\delta$-scrambled set for some $\delta>0$.

\begin{proof}
Let $\{x,y\}$ be a scrambled pair of $f$. At least one of the points $x,y$ has an infinite $\omega$-limit set, say $\omega_f(x)$ is infinite. Then $\omega_f(x)$, being compact and countable, has an isolated point (in the topology of $\omega_f(x)$) and being also invariant for $f$, by Lemma~\ref{L:ppp} contains a periodic point. Now apply Proposition~\ref{P:special omega}.
\end{proof}

\bibliographystyle{plain}

\begin{thebibliography}{10}

\bibitem{ADR}
Ll. Alsed{\`a}, M.~A.~Del R{\'{\i}}o, and J.~A. Rodr{\'{\i}}guez.
\newblock Transitivity and dense periodicity for graph maps.
\newblock {\em J. Difference Equ. Appl.}, 9(6):577--598, 2003.

\bibitem{BDM}
F.~Blanchard, F.~Durand, and A.~Maass.
\newblock Constant-length substitutions and countable scrambled sets.
\newblock {\em Nonlinearity}, 17(3):817--833, 2004.

\bibitem{BGKM}
F.~Blanchard, E.~Glasner, S.~Kolyada, and A.~Maass.
\newblock On {L}i-{Y}orke pairs.
\newblock {\em J. Reine Angew. Math.}, 547:51--68, 2002.

\bibitem{BHuS}
F.~Blanchard, W.~Huang, and \mL. Snoha.
\newblock Topological size of scrambled sets.
\newblock {\em Colloq. Math.}, 110(2):293--361, 2008.

\bibitem{BC}
L.~S. Block and W.~A. Coppel.
\newblock {\em Dynamics in one dimension}.
\newblock Lecture Notes in Mathematics, no. 1513. Springer-Verlag, 1992.

\bibitem{B-book}
A.~M. Blokh.
\newblock ``{S}pectral expansion''\ for piecewise monotone mappings of an
  interval.
\newblock {\em Uspekhi Mat. Nauk}, 37(3-225):175--176, 1982.
\newblock (Russian). English translation: {\it Russ. Math. Surv.},
  37(3):198--199, 1982.

\bibitem{B0}
A.~M. Blokh.
\newblock On transitive mappings of one-dimensional branched manifolds.
\newblock In {\em Differential-difference equations and problems of
  mathematical physics}, pages 3--9, 131. Akad. Nauk Ukrain. SSR Inst. Mat.,
  Kiev, 1984.
\newblock (Russian).

\bibitem{B1}
A.~M. Blokh.
\newblock Dynamical systems on one-dimensional branched manifolds. {I}.
\newblock {\em Teor. Funktsi\u\i\ Funktsional. Anal. i Prilozhen.}, (46):8--18,
  1986.
\newblock (Russian). English translation in {\it J. Soviet Math.},
  48(5):500--508, 1990.

\bibitem{B2}
A.~M. Blokh.
\newblock Dynamical systems on one-dimensional branched manifolds. {II}.
\newblock {\em Teor. Funktsi\u\i\ Funktsional. Anal. i Prilozhen.},
  (47):67--77, 1987.
\newblock (Russian). English translation in {\it J. Soviet Math.},
  48(6):668--674, 1990.

\bibitem{B3}
A.~M. Blokh.
\newblock Dynamical systems on one-dimensional branched manifolds. {III}.
\newblock {\em Teor. Funktsi\u\i\ Funktsional. Anal. i Prilozhen.},
  (48):32--46, 1987.
\newblock (Russian). English translation in {\it J. Soviet Math.},
  49(2):875--883, 1990.

\bibitem{B4}
A.~M. Blokh.
\newblock On the connection between entropy and transitivity for
  one-dimensional mappings.
\newblock {\em Russ. Math. Surv.}, 42(5):165--166, 1987.

\bibitem{Bor}
K.~Borsuk.
\newblock {\em Theory of retracts}.
\newblock Monografie Matematyczne, Tom 44. Pa\'nstwowe Wydawnictwo Naukowe,
  Warsaw, 1967.

\bibitem{FPS}
G.~L. Forti, L.~Paganoni, and J.~Sm{\'{\i}}tal.
\newblock Strange triangular maps of the square.
\newblock {\em Bull. Austral. Math. Soc.}, 51(3):395--415, 1995.

\bibitem{GL}
J.-L.~Garc{\'{\i}}a Guirao and M.~Lampart.
\newblock Li and {Y}orke chaos with respect to the cardinality of the scrambled
  sets.
\newblock {\em Chaos Solitons Fractals}, 24(5):1203--1206, 2005.

\bibitem{HM}
R.~Hric and M.~M{\'a}lek.
\newblock Omega limit sets and distributional chaos on graphs.
\newblock {\em Topology Appl.}, 153(14):2469--2475, 2006.

\bibitem{HY}
W.~Huang and X.~Ye.
\newblock Homeomorphisms with the whole compacta being scrambled sets.
\newblock {\em Ergod. Th. Dynam. Systems}, 21(1):77--91, 2001.

\bibitem{HY2}
W.~Huang and X.~Ye.
\newblock Devaney's chaos or 2-scattering implies {L}i-{Y}orke's chaos.
\newblock {\em Topology Appl.}, 117(3):259--272, 2002.

\bibitem{JaSm}
K.~Jankov{\'a} and J.~Sm{\'\i}tal.
\newblock A characterization of chaos.
\newblock {\em Bull. Austral. Math. Soc.}, 34(2):283--292, 1986.

\bibitem{K}
M.~Kuchta.
\newblock Characterization of chaos for continuous maps of the circle.
\newblock {\em Comment. Math. Univ. Carolin.}, 31(2):383--390, 1990.

\bibitem{KS}
M.~Kuchta and J.~Sm{\'\i}tal.
\newblock Two-point scrambled set implies chaos.
\newblock In {\em European Conference on Iteration Theory (Caldes de Malavella,
  1987)}, pages 427--430. World Sci. Publishing, Teaneck, NJ, 1989.

\bibitem{LM}
J.~Llibre and M.~Misiurewicz.
\newblock Horseshoes, entropy and periods for graph maps.
\newblock {\em Topology}, 32(3):649--664, 1993.

\bibitem{MS2}
J.-H. Mai and S.~Shao.
\newblock Spaces of {$\omega$}-limit sets of graph maps.
\newblock {\em Fund. Math.}, 196(1):91--100, 2007.

\bibitem{MS}
J.-H. Mai and S.~Shao.
\newblock The structure of graph maps without periodic points.
\newblock {\em Topology Appl.}, 154(14):2714--2728, 2007.

\bibitem{R}
M.~Rees.
\newblock A minimal positive entropy homeomorphism of the {$2$}-torus.
\newblock {\em J. London Math. Soc. (2)}, 23(3):537--550, 1981.

\bibitem{Sar}
P.~V. S.~P. Saradhi.
\newblock {\em Sets of periods of continuous self maps on some metric spaces}.
\newblock PhD thesis, University of Hyderabad, 1987.

\bibitem{Shark}
A.~N. Sarkovskii.
\newblock On attracting and attracted sets.
\newblock {\em Dokl. Akad. Nauk SSSR}, 160:1036--1038, 1965.
\newblock (Russian).

\bibitem{Sm}
J.~Sm{\'\i}tal.
\newblock Chaotic functions with zero topological entropy.
\newblock {\em Trans. Amer. Math. Soc.}, 297(1):269--282, 1986.

\bibitem{St}
S.~T. Stefanov.
\newblock Problem no. 10476.
\newblock {\em Amer. Math. Monthly}, 102(8):746, 1995.
\newblock Solution in {\it Amer. Math. Monthly}, 105(4): 370, 1998.

\bibitem{We}
B.~Weiss.
\newblock Topological transitivity and ergodic measures.
\newblock {\em Math. Systems Theory}, 5:71--75, 1971.

\end{thebibliography}

\end{document}